\documentclass[12pt]{article}
\usepackage{amsmath, amsthm, amssymb,enumitem,multirow,rotating,longtable,slashbox,hhline}

\newtheorem{prop}{Proposition}[section]

\newtheorem{cor}[prop]{Corollary}\begin{document}
\begin{center}
{\large \bf On Two OEIS Conjectures}\\
Jeremy M. Dover
\end{center}
\begin{abstract}
In~\cite{stephan}, Stephan enumerates a number of conjectures regarding
integer sequences contained in Sloane's On-line Encyclopedia of Integer
Sequences~\cite{oeis}. In this paper, we prove two of these conjectures.
\end{abstract}

\section{Proof of Conjecture 110}

In~\cite{stephan} Stephan formulates the following conjecture:
\begin{quote}
Define $a_n = |\{(i,j): 0 \le i,j < n \; {\rm and} \; i \; {\rm AND} \; j > 0\}|$,
where AND is the bitwise and operator. Then $a_n$ is the sequence given by the 
recursions $a_{2n} = 3a_n + n^2$, $a_{2n+1} = a_n +2a_{n+1} + n^2-1$, with initial 
conditions $a_0 = a_1 = 0$.
\end{quote}

Truth for the initial conditions is easy to calculate. In what follows, we define $S_n
= \{(i,j): 0 \le i,j < n \; {\rm and} \; i \; {\rm AND} \; j > 0\}$ from which we have
$a_n = |S_n|$.

To see that $a_{2n} = 3a_n + n^2$, we partition the set $S_{2n}$ into four parts, $EE_{2n}$,
$EO_{2n}$, $OE_{2n}$ and $OO_{2n}$ where $EE_{2n} = \{(x,y) \in S_{2n}: x,y \; {\rm even}\}$ 
and the other sets are defined analogously. We count the number of elements in each set.

Since the AND of two odd numbers is always at least 1, $OO_{2n}$ consists of all pairs of
odd numbers $(x,y)$ with $0 \le x,y < n$, of which there are $n^2$ possibilities. To count
the number of elements in $EE_{2n}$, let $(2i,2j)$ be an element of $EE_{2n}$ which forces
$0 \le i,j < n$. We know $2i$ AND $2j$ is nonzero, and since the low order bit of both 
numbers is zero, we must have $i$ AND $j$ nonzero. Thus for each $(2i,2j) \in EE_{2n}$ we 
have the corresponding pair $(i,j) \in S_n$. Moreover for any pair $(i,j) \in S_n$ it is
easy to see that $(2i,2j) \in S_{2n}$. Therefore $|EE_{2n}| = |S_n| = a_n$.

As with the previous case we can write each element of $OE_{2n}$ as $(2i+1,2j)$ for $0 \le 
i,j < n$ where $2i+1$ AND $2j$ is nonzero. Since the low order bit of $2j$ is zero, we must
have $i$ AND $j$ nonzero, thus $(i,j) \in S_n$. Then for any $(i,j) \in
S_n$ we have $(2i+1,2j) \in S_{2n}$, showing that $|OE_{2n}| = a_n$. Since $|EO_{2n}|$ is 
obviously equal to $|OE_{2n}|$, we find that $a_{2n} = 3a_n + n^2$ as conjectured.

To show $a_{2n+1} = a_n + 2a_{n+1} + n^2 - 1$ we use the same sort of analysis, but one
of the counts is more tricky. Using the same definitions for $OO$, et. al., the exact same
argument as above shows that $|OO_{2n+1}| = n^2$. For $EE_{2n+1}$ we again look at elements
$(2i,2j) \in S_{2n+1}$ and still must have $i$ AND $j$ nonzero, but now $0 \le i,j < n+1$
since either $i$ and/or $j$ can be $n$. Thus $|EE_{2n+1}| = a_{n+1}$, as opposed to $a_n$
for the previous case.

Counting $|OE_{2n+1}|$ is more tricky in this case. Again we have that the elements of 
$OE_{2n+1}$ are the pairs $(2i+1,2j)$ such that $i$ AND $j$ is nonzero, but in this case 
$0 \le i < n$, while $0 \le j < n+1$. Thus $|OE_{2n+1}|$ is the number of elements of 
$(i,j) \in S_{n+1}$ for which $i \neq n$. To determine this number we provide a different
partition of $S_{n+1}$ as:
\begin{eqnarray*}
S_{n+1} & = &\{(i,j) \in S_{n+1}: i,j < n\} \cup \{(n,j) \in S_{n+1}: j < n\} \cup \\ 
& & \{(i,n) \in S_{n+1}: i < n\} \cup \{(n,n)\}
\end{eqnarray*}

In this expression, the set $\{(i,j) \in S_{n+1}:i,j < n\}$ is exactly $S_n$. If we let $x$
be the cardinality of $\{(n,j) \in S_{n+1}:j < n\}$, then by taking cardinalities of each
partition element we have
\begin{displaymath}
a_{n+1} = a_n + 2x + 1
\end{displaymath}

Therefore, $x = \frac12(a_{n+1} - a_n -1)$ and we have $|OE_{2n+1}| = a_n + \frac12(a_{n+1} - a_n -1)
= \frac12(a_{n+1}+a_n-1)$. Since $EO_{2n+1}$ has the same cardinality we finally have
\begin{eqnarray*}
a_{2n+1} & = & n^2 + a_{n+1} + \frac12(a_{n+1}+a_n-1) + \frac12(a_{n+1}+a_n-1) \\
& = & a_n + 2a_{n+1} +n^2-1
\end{eqnarray*}
which finishes the conjecture.

\section{Proof of Conjecture 115}

In~\cite{stephan} Stephan formulates the following conjecture:
\begin{quote}
Define the sequence $a_n$ by $a_1=1$ and $a_n = M_n + m_n$, where $M_n = 
{\rm max}_{1 \le i < n} (a_i + a_{n-i})$ and $m_n = {\rm min}_{1 \le i < n} 
(a_i + a_{n-i})$. Let further $b_n$ be the number of binary partitions of $2n$
into powers of 2 (number of {\em binary partitions}). Then
\end{quote}
\begin{displaymath}
m_n = \frac{3}{2}b_{n-1}-1,\; M_n = n + \sum_{k=1}^{n-1} m_n, \; a_n = M_{n+1} -1
\end{displaymath}

We prove this conjecture through a series of short induction proofs. For 
convenience we define $m_1 = 1$.

\begin{prop}
\label{prop1}
The sequence $a_n$ is strictly increasing, and thus positive for all $n \ge 1$.
Moreover, each of the sequences $M_n$ and $m_n$ is also strictly increasing
and positive for all $n \ge 2$.
\end{prop}

\begin{proof}
It suffices to show that $a_n > a_{n-1}$, $M_n > M_{n-1}$, and
$m_n > m_{n-1}$ for all $n \ge 3$, and we proceed by induction on $n$. By 
definition $a_1 = 1$ and $a_2$ is easily computed to be 4, proving the base
for our induction, as well as that $a_2 > a_1$.

To complete our induction step, we assume the result is true for all $a_i$
with $i < n$, and consider $a_n$. By definition $M_n$ and $m_n$ are the maximum
and minimum, respectively, of the set $S_n = \{a_i + a_{n-i}: 1 \le i < n\}$. 
By our induction hypothesis each of the $a_i$'s is positive, implying that
$M_n$ and $m_n$ are positive as well. Now $a_n = M_n + m_n > M_n$. Since
$a_{n-1} + a_1 = a_{n-1} + 1$ is in the set $S$, we must have $M_n \ge a_{n-1} +1$. Thus $a_n > M_n \ge a_{n-1} + 1$, showing that the sequence $a_n$ is
increasing.

To show that $M_n$ is increasing, suppose that the maximum value in $S_n$ is
given by the element $a_i+a_{n-i}$. Then $S_{n+1}$ contains the element 
$a_i + a_{n+1-i}$ which is strictly greater than $a_i + a_{n-i}$ since $a_i$ is
increasing. Thus the maximum value in $S_{n+1}$ must be larger than the 
greatest value in $S_n$, proving $M_n$ is increasing.

A similar argument shows $m_n$ is increasing. Now let $a_i + a_{n+1-i}$ be 
the smallest element in $S_{n+1}$, and thus equal to $m_{n+1}$. Then $S_n$
must contain the element $a_i + a_{n-i}$ which is strictly less than 
$m_{n+1}$. Then we have $m_n \le a_i + a_{n-i} < m_{n+1}$, showing $m_n$ is
increasing.
\end{proof}

\begin{prop}
\label{prop2}
$M_{n+1} = a_n + 1$ for all $n \ge 1$.
\end{prop}

\begin{proof}
We again proceed by induction on $n$, noting that the base case
$M_2 = a_1 + 1 = 2$ is easily calculated. For our strong induction hypothesis
assume $M_{i+1} = a_i +1$ for all $1 \le i < n$, and attempt to prove the
result $M_{n+1} = a_n + 1$.

By the definition of $a_n$ we have $a_n = M_n + m_n$, from which $a_n + 1 =
M_n + m_n + 1$. Using our induction hypothesis we know $M_n = a_{n-1} + 1$ from
which we obtain $a_n + 1 = a_{n-1} + m_n + 2$. Iterating this procedure of
alternately applying the definition of $a_i$ and the induction hypothesis 
gives:
\begin{equation}
\label{basicrelation}
a_n + 1 = a_i + \sum_{j=i+1}^{n} m_i + (n-i+1)
\end{equation}
for all $1 \le i < n$.

To show $M_{n+1} = a_n+1$ we must show that $a_n + 1 \ge a_i + a_{n+1-i}$ 
for all $2 \le i < n$. By symmetry in the indices of the sequence this is 
equivalent to showing this result for all $\lceil \frac{n+1}{2} \rceil \le i 
< n$.

Since $a_i$ is positive and increasing for all $i \ge 1$, 
$m_n > a_i$ for all $1 \le i \le \lceil \frac{n}{2} \rceil$ as every 
element in $S_n = \{a_i + a_{n-i}: 1 \le i < n\}$ has a summand $a_j$ for 
which $j \ge \lceil \frac{n}{2} \rceil$.

From Equation~\ref{basicrelation} we know that $a_n + 1 = a_i + 
\sum_{j=i+1}^{n} m_j + (n-i+1) > a_i + m_n$ for all $1 \le i < n$. This
implies $a_n + 1 > a_i + a_{n+1-i}$ for all $\lfloor \frac{n}{2} \rfloor + 1 
\le i < n$. As $\lfloor \frac{n}{2} \rfloor + 1 = \lceil \frac{n+1}{2} \rceil$
for all positive integers $n$, this proves the claim.
\end{proof}

\begin{cor}
\label{cor3}
$M_n = \sum_{k=1}^{n-1} m_k + n$ for all $n \ge 1$.
\end{cor}

\begin{proof}
This follows immediately from Equation~\ref{basicrelation} with
$i = 1$.
\end{proof}

The conjectured value for $m_n$ is somewhat trickier, and requires the 
following intermediate result.

\begin{prop}
\label{prop4}
$m_n = a_{\lfloor \frac{n}{2} \rfloor} + a_{\lceil \frac{n}{2} \rceil}$ for 
all $n \ge 2$.
\end{prop}

\begin{proof}
By the definition of $m_n$, $m_n$ is the smallest element of
$S_n = \{a_i + a_{n-i}: 1 \le i < n\}$. Using the results of 
Proposition~\ref{prop2} and Corollary~\ref{cor3}, we can calculate:
\begin{eqnarray*}
a_i + a_{n-i} & = & M_{i+1} - 1 + M_{n+1-i} - 1 \\
& = & \sum_{k=1}^{i} m_k + (i+1) + \sum_{k=1}^{n-i} m_k + (n+1-i) - 2\\
& = & n+\sum_{k=1}^{i} m_k +\sum_{k=1}^{n-i} m_k
\end{eqnarray*}

Denoting $d_i = a_i + a_{n-i}$ and noting that $d_i = d_{n-i}$, we need to show
that $d_{\lfloor \frac{n}{2} \rfloor} \le d_i$ for all $1 \le i < \lfloor
\frac{n}{2} \rfloor$. To do this, we will show that $d_i$ is a decreasing 
sequence in the interval $1 \le i \le \lfloor \frac{n}{2} \rfloor$ by looking
at the differences $d_i - d_{i+1}$ for $1 \le i < \lfloor \frac{n}{2} \rfloor$
and showing they are all positive.

To this end we calculate
\begin{eqnarray*}
d_i - d_{i+1} & = & n + \sum_{k=1}^i m_k + \sum_{k=1}^{n-i} m_k - \left[ 
n + \sum_{k=1}^{i+1} m_k + \sum_{k=1}^{n-(i+1)} m_k \right] \\
& = & m_{n-i} - m_{i+1}
\end{eqnarray*}
By Proposition~\ref{prop1}, $m_n$ is an increasing sequence, and $n-i > i+1$
for all $1 \le i < \lfloor \frac{n}{2} \rfloor$, which implies that $d_i -
d_{i+1}$ is positive for all $1 \le i < \lfloor \frac{n}{2} \rfloor$, proving
the result.
\end{proof}

Prior to proving the final piece of the conjecture, we need the following
result from Sloane~\cite{oeis}. The sequence $b_n$, denoted A000123 by Sloane,
where $b_n$ is the number of partitions of $2n$ into powers of 2 satisfies the 
recursion $b_n = b_{n-1} + b_{\lfloor \frac{n}{2} \rfloor}$, for all $n \ge 2$
with initial condition $b_1 = 2$.

\begin{prop}
$m_n = \frac{3}{2} b_{n-1} - 1$ for all $n \ge 2$.
\end{prop}

\begin{proof}
To show that this equation holds, we prove that the sequence
$c_n=\frac{2}{3}(m_{n+1} + 1)$ satisfies the recursion and initial conditions 
for $b_n$. Rather than use the form $b_n = b_{n-1} + b_{\lfloor \frac{n}{2} 
\rfloor}$ for the recursion, we will instead use $b_n - b_{n-1} = b_{\lfloor 
\frac{n}{2} \rfloor}$, being easier for computations.

Using Proposition~\ref{prop4}, we have
\begin{displaymath}
c_n - c_{n-1}  =  \frac23 \left( a_{\lfloor \frac{n+1}{2} \rfloor} + 
a_{\lceil \frac{n+1}{2} \rceil} + 1 \right) - \frac23 \left( a_{\lfloor 
\frac{n}{2} \rfloor} + a_{\lceil \frac{n}{2} \rceil} + 1 \right)
\end{displaymath}
Noting that $\lceil \frac{n}{2} \rceil = \lfloor \frac{n+1}{2} \rfloor$ and
$\lceil \frac{n+1}{2} \rceil = \lfloor \frac{n}{2} \rfloor + 1$ for all
integers $n \ge 1$, this reduces to
\begin{displaymath}
c_n-c_{n-1} = \frac23 \left( a_{\lfloor \frac{n}{2} \rfloor + 1} - a_{\lfloor
\frac{n}{2} \rfloor} \right)
\end{displaymath}

Now using the formula for $a_i$ given in Proposition~\ref{prop2} and 
Corollary~\ref{cor3}, we obtain
\begin{displaymath}
c_n-c_{n-1} = \frac23 \left( \lfloor \frac{n}{2} \rfloor + 2 + 
\sum_{k = 1}^{\lfloor \frac{n}{2} \rfloor + 1} m_k - \left[ \lfloor
\frac{n}{2} \rfloor + 1 + \sum_{k=1}^{\lfloor \frac{n}{2} \rfloor} m_k 
\right] \right)
\end{displaymath}
which reduces to
\begin{eqnarray*}
c_n-c_{n-1} & = & \frac23 \left( m_{\lfloor \frac{n}{2} + 1\rfloor} + 1\right)\\
& = & c_{\lfloor \frac{n}{2} \rfloor}
\end{eqnarray*}

Therefore $c_n = \frac23 (m_n+1)$ satsfies the same recursion as $b_n$, and
$c_1 = \frac23 (m_2 + 1) = 2 = b_1$, implying that $b_n = \frac23 (m_n+1)$ for
all $n \ge 2$, from which we obtain the result.
\end{proof}

\bibliographystyle{plain}

\end{document}